\documentclass[12pt]{article}
\usepackage{mathrsfs}
\usepackage{amsmath,amsfonts,amssymb,amsthm}

\def\<{\langle}
\def\>{\rangle}

\def\g{\mathfrak g}

\newtheorem{thm}{Theorem}[section]
\newtheorem{prop}[thm]{Proposition}
\newtheorem{lem}[thm]{Lemma}

\newtheorem{rmk}[thm]{Remark}
\newtheorem{definition}[thm]{Definition}

\begin{document}

\begin{center}
{\Large \bf  Some categories of modules for toroidal Lie algebras }
\end{center}

\begin{center}
{Hongyan Guo, Shaobin Tan\footnote{Partially supported by NSF of
China (No.10931006) and a grant from the PhD Programs Foundation of
Ministry of Education of China (No.20100121110014).} and Qing
Wang\footnote{Partially supported by NSF of China (No.11371024),
Natural Science Foundation of Fujian Province (No.2013J01018) and
Fundamental Research Funds for the Central
University (No. 2013121001).}\\
School of Mathematical Sciences, Xiamen University, Xiamen 361005,
China}
\end{center}

\begin{abstract}
 In this paper, we use basic formal variable techniques to
study certain categories of modules for the toroidal Lie algebra
$\tau$. More specifically, we define and study two categories
$\mathcal{E}_{\tau}$ and $\mathcal{C}_{\tau}$ of $\tau$-modules
using generating functions, where $\mathcal{E}_{\tau}$ is proved to
contain the evaluation modules while $\mathcal{C}_{\tau}$ contains
certain restricted $\tau$-modules, the evaluation modules, and their
tensor product modules. Furthermore, we classify the irreducible
integrable modules in categories $\mathcal{E}_{\tau}$ and
$\mathcal{C}_{\tau}$.
\end{abstract}

\section{Introduction}
\def\theequation{1.\arabic{equation}}
\setcounter{equation}{0}

 Let $\hat{\mathfrak{g}}={\mathfrak{g}}\otimes {\mathbb{C}}[t,t^{-1}]\oplus
 {\mathbb{C}}{\mathbf{k}}$ be
the untwisted affine Lie algebra of a finite
dimensional simple Lie algebra $\mathfrak{g}$, and let
$\widetilde{\mathfrak{g}}={\mathfrak{g}}\otimes
{\mathbb{C}}[t,t^{-1}]\oplus {\mathbb{C}}{\mathbf{k}}\oplus
 {\mathbb{C}}d$ be the affine Lie algebra with the derivation added.
 The irreducible integrable modules with finite dimensional weight spaces
for $\widetilde{\mathfrak{g}}$ were classified by Chari \cite{C}. In
particular, Chari proved that every irreducible integrable
$\hat{\mathfrak{g}}$-module of level zero with finite dimensional
weight spaces must be a finite dimensional evaluation module.
Furthermore, in [CP], they studied tensor products of modules of
different types which were obtained in \cite{C}. In particular, they
proved that the tensor product of
  an irreducible integrable highest
 $\hat{\mathfrak{g}}$-module and a finite dimensional irreducible
 evaluation
  $\hat{\mathfrak{g}}$-module is irreducible. Such irreducible tensor product modules
 are integrable but with infinite dimensional weight spaces in general. This provides a new family
  of irreducible integrable $\hat{\mathfrak{g}}$-modules.
 Motivated by this, a
canonical characterization for the tensor product modules of this
form was given in \cite{L1}. More specifically, two categories $\mathcal{E}$
and $\mathcal{C}$ of $\hat{\mathfrak{g}}$-modules were defined and
studied by using generating functions and formal variables in
\cite{L1}, where the category $\mathcal{C}$ unifies highest weight
modules, evaluation modules, and their tensor product modules.
Furthermore, all irreducible integrable $\hat{\mathfrak{g}}$-modules
in categories $\mathcal{E}$ and $\mathcal{C}$ were classified in
\cite{L1}.

Toroidal Lie algebras, which are central extensions of multi-loop
Lie algebras, are natural generalizations of affine Kac-Moody Lie algebras.
For toroidal Lie algebras, irreducible integrable modules with
finite-dimensional weight spaces have been classified and
constructed in \cite{R}. It was proved therein that there are two
classes of evaluation modules depending on the
action of the center. One class consists of the evaluation modules arising from
finite dimensional irreducible modules for
 the finite dimensional simple Lie algebra with the full center acting trivially,
 while the other consists of the evaluation modules arising from irreducible integrable highest weight modules for the affine Lie
algebra with a nontrivial action of the center. Then a natural question is whether one can
apply the formal variable technique used \cite{L1} to study various categories of
modules for toroidal Lie algebras. This is the main motivation of this current paper.

In this paper, we define two categories $\mathcal{E}_{\tau}$ and
          $\mathcal{C}_{\tau}$ of modules for the toroidal Lie algebra $\tau$.
It is shown that the category $\mathcal{E}_{\tau}$ contains
evaluation modules. As for the modules in category
$\mathcal{C}_{\tau}$, we need a notion of restricted $\tau$-module,
defined in \cite{LTW}. In \cite{LTW}, a theory of
toroidal vertex algebras and their modules was developed,
 and we associated toroidal vertex algebras and their modules to toroidal Lie algebras,
 and proved that there exists a canonical correspondence between
 restricted modules for toroidal Lie algebras and modules for toroidal vertex
 algebras. In view of this, restricted modules for toroidal Lie algebras are also
 important in the study of representation theory of toroidal vertex
 algebras. The category $\mathcal{C}_{\tau}$ defined in this paper contains the evaluation modules, certain restricted
$\tau$-modules where we denote the corresponding module category by
$\mathcal{\widetilde{R}}$,
 and their tensor product modules.

 Note that the category $\mathcal{\widetilde{R}}$ is a mixed product of the notion of restricted
 modules and the notion of evaluation modules for an affine Lie algebra, which is significantly different from the category $\mathcal{R}$
 of restricted modules for an affine Lie algebra. Due to this, the treatment for the classification of irreducible integrable $\tau$-modules
 in $\mathcal{\widetilde{R}}$ is more complicated than that for the case of the classification of irreducible
  integrable $\hat{\mathfrak{g}}$-modules in $\mathcal{R}$.
 In the later case,  we know from \cite{DLM} that every nonzero restricted integrable $\hat{\mathfrak{g}}$-module is
 a direct sum of irreducible highest weight integrable modules. While in our case with $\mathcal{\widetilde{R}}$,
  we have no such good result for restricted integrable modules for toroidal Lie algebras
  $\tau$. Thus
 a new method is needed to deal with this problem. As our main results, we prove that irreducible integrable $\tau$-modules in the
category $\mathcal{E}_{\tau}$ are exactly the evaluation modules
arising from finite-dimensional irreducible $\mathfrak{g}$-modules
and that irreducible integrable $\tau$-modules in category
$\mathcal{\widetilde{R}}$ are exactly the evaluation modules arising
from irreducible integrable highest weight modules for the affine
Lie algebra. In this way, the two classes of modules which were
classified and constructed in \cite{R} are unified through
irreducible integrable modules in category $\mathcal{C}_{\tau}$.

This paper is organized as follows. In Section 2, we review some
basic notions and facts about the toroidal Lie algebras. In Section 3,
we define the category $\mathcal{E}_{\tau}$ and
classify all irreducible integrable modules in this
category. In Section 4, we define the category
$\mathcal{C}_{\tau}$ and classify all irreducible
integrable modules in $\mathcal{C}_{\tau}$.

\section{Preliminaries}
\label{Sect:V(k,0)}
\def\theequation{2.\arabic{equation}}
\setcounter{equation}{0}

In this section, we recall the definitions of restricted modules and
integrable modules for the toroidal Lie algebra. We also present
some basic facts about them.

Throughout this paper, the symbols $x,x_{0},x_{1},x_{2},\dots$ will
denote mutually commuting independent formal variables. All vector
spaces are considered to be over $\mathbb{C}$, the field of complex
numbers. For a vector space $U$, $U((x))$ is the vector space of
lower truncated integral power series in $x$ with coefficients
 in $U$, $U[[x]]$ is the vector space of nonnegative integral
 power series in $x$ with coefficients in $U$, and
$U[[x,x^{-1}]]$ is the vector space of doubly infinite integral
 power series of $x$ with coefficients in $U$.
 Multi-variable analogues of these vector spaces are defined
  in the same way.

  Recall from \cite{FLM} the formal delta function
  $$\delta(x)=\sum\limits_{n\in \mathbb{Z}}x^{n}\in
  {\mathbb{C}}[[x,x^{-1}]]$$ and the basic property
  $$f(x)\delta\left(\frac{a}{x}\right)=f(a)\delta\left(\frac{a}{x}\right)$$ for
  any $f(x)\in{\mathbb{C}}[x,x^{-1}]$ and for any nonzero complex number
  $a$.

Let $\mathfrak{g}$ be a finite dimensional simple Lie algebra
equipped with the normalized Killing form $\langle\cdot,\cdot \rangle$. Let
$r$ be a positive integer and let $\mathbb{C}[t_{1}^{\pm
1},\dots,t_{r}^{\pm 1}]$ be the algebra of Laurent polynomials in
$r$ commuting variables. For $\underline{m} =
(m_{1},\dots,m_{r})\in\mathbb{Z}^{r}$, set $t^{\underline{m}} =
t_{1}^{m_{1}}\cdots t_{r}^{m_{r}}$. Define a Lie algebra
\begin{eqnarray}
\tau=(\mathfrak{g}\otimes\mathbb{C}[t_{0},
t_{0}^{-1}]\oplus\mathbb{C}K_0)
  \otimes\mathbb{C}[t_{1}^{\pm 1},\dots,t_{r}^{\pm 1}]\oplus\sum\limits_{i=1}^{r}\mathbb{C}K_i,
  \end{eqnarray}
where
\begin{eqnarray}
 [a\otimes
t_{0}^{n_{0}}t^{\underline{n}}, b\otimes
t_{0}^{m_{0}}t^{\underline{m}}]& = & [a,b]\otimes
t_{0}^{n_{0}+m_{0}}t^{\underline{n}+\underline{m}}+n_{0}\langle a,
b\rangle \delta_{n_{0}+m_{0},0}K_{0}\otimes
 t^{\underline{n}+\underline{m}}
                                            \nonumber\\
 &&{} +\langle a, b\rangle \delta_{n_{0}+m_{0},0}
  \delta_{\underline{n}+\underline{m},0}\sum_{i=1}^{r}n_{i}K_{i}
 \label{eq:2.1}\end{eqnarray}
for $a,b\in \mathfrak{g},
 n_{0},m_{0}\in\mathbb{Z},\underline{n},\underline{m}\in\mathbb{Z}^{r},$
 and where $K_0\otimes t^{\underline{m}}$, $K_i$, $i=1,\dots,r$ are
 central. The Lie algebra $\tau$ is the
 toroidal Lie algebra which we consider in this paper. {\em From now on, $\tau$ always stands for
 this particular toroidal Lie algebra}.

 For $a\in
 \g$, we denote $a(n_{0},\underline{n})= a\otimes
t_{0}^{n_{0}}t^{\underline{n}}$. For $K_0$, we denote
$K_0(\underline{n})=K_0\otimes t^{\underline{n}}$. Then the Lie
bracket relation (\ref{eq:2.1}) amounts to

\begin{eqnarray}
 [a(n_{0},\underline{n}),b(m_{0},\underline{m})]& =
& [a,b](n_{0}+m_{0},\underline{n}+\underline{m})+n_{0}\langle a ,
b\rangle \delta_{n_{0}+m_{0},0}K_{0}(\underline{n}+\underline{m})
                                            \nonumber\\
 &&{} +\langle a , b\rangle \delta_{n_{0}+m_{0},0}
  \delta_{\underline{n}+\underline{m},0}\sum_{i=1}^{r}n_{i}K_{i}. \label{eq:2.2}
 \end{eqnarray}
For $a\in \mathfrak{g}$, form a generating function
$$ a(x_{0},\underline{x}) =
\sum_{n_{0}\in \mathbb{Z}}\sum_{\underline{n}\in\mathbb{Z}^{r}}
a(n_{0},\underline{n})x_{0}^{-n_{0}-1}x^{-\underline{n}-1}
     \in\tau[[x_{0}^{\pm 1},x_{1}^{\pm 1},\dots,x_{r}^{\pm 1}]] $$
and set
$$ K_{0}(\underline{x}) = \sum_{\underline{n}\in\mathbb{Z}^{r}}
K_{0}(\underline{n})x^{-\underline{n}-1}\in \tau[[x_{1}^{\pm
1},\dots,x_{r}^{\pm 1}]],$$ where $x^{-\underline{n}-1} =
x_{1}^{-n_{1}-1}\cdots x_{r}^{-n_{r}-1}.$ In terms of the generating
functions, the defining relations (\ref{eq:2.1}) are equivalent to
\begin{eqnarray}
&[a(x_{0},\underline{x}),b(y_{0},\underline{y})]   =
[a,b](y_{0},\underline{y})\prod\limits_{i=0}^{r}y_{i}^{-1}\delta(\frac{x_{i}}{y_{i}})
+\langle a , b\rangle K_{0}(\underline{y}) (\frac{\partial}{\partial
y_{0}}y_{0}^{-1}\delta(\frac{x_{0}}{y_{0}}))
\prod\limits_{i=1}^{r}y_{i}^{-1}\delta(\frac{x_{i}}{y_{i}}) \nonumber\\
&{}+\langle a , b\rangle
x_{0}^{-1}y_{0}^{-1}\delta(\frac{x_{0}}{y_{0}})
\sum\limits_{j=1}^{r}(K_{j}(\frac{\partial}{\partial
y_{j}}y_{j}^{-1}\delta(\frac{x_{j}}{y_{j}}))\prod\limits_{i=1,i\neq
j}^{r}x_{i}^{-1}y_{i}^{-1}\delta(\frac{x_{i}}{y_{i}})).
\label{eq:2.3}
\end{eqnarray}

\begin{definition}\label{t1} A $\tau$-module $W$ is said to be restricted
       if for any $w\in W,a\in \mathfrak{g}$,
       $$a(n_{0},\underline{n})w = 0 $$
       for $n_{0}$ sufficiently large, or equivalently
       $$a(x_{0},\underline{x})w\in
                   W[[x_{1}^{\pm 1},\dots,x_{r}^{\pm 1}]]((x_{0})).$$
\end{definition}

Let $\mathfrak{h}$ be a Cartan subalgebra of $\g$ and let
$\underline{\mathfrak{h}}$ be the linear span of $\mathfrak{h}$,
$K_{0}(\underline{n})$ and $K_{i}$ for
$\underline{n}\in\mathbb{Z}^{r}, 1\leq i\leq r$. We denote
by $\underline{\mathfrak{h}}^{*}$ the dual space of
$\underline{\mathfrak{h}}$. Let $\Delta$ be the root system
of $\mathfrak{g}$ and for each root $\alpha\in\Delta$, we fix a nonzero root vector $x_{\alpha}$.

\begin{definition}\label{} A $\tau$-module $W$  is called integrable if

    (i)  $W = \bigoplus_{\lambda\in\underline{\mathfrak{h}}^{\ast}} W_{\lambda}$,
    where $W_{\lambda} = \{w\in W\mid hw = \lambda(h)w,\forall h \in
    \underline{\mathfrak{h}}\}$.

    (ii) for any $\alpha\in \Delta$, $m_{0}\in\mathbb{Z}$, $\underline{m}\in \mathbb{Z}^{r}$
    and $w\in W$, there exists a nonnegative integer $k = k(\alpha,m_{0},\underline{m},w)$
    such that $(x_{\alpha}(m_{0},\underline{m}))^{k}w = 0$.
\end{definition}

The following lemma is similar to Lemma 2.3 in \cite{L1}.

\begin{lem}\label{lem:2.3}  There exists a basis $\{a_{1},\dots,a_{l}\}$ of
$\mathfrak{g}$ such that
$[a_{i}(n_{0},\underline{n}),a_{i}(m_{0},\underline{m})] = 0$ for
$1\leq i\leq
l,m_{0},n_{0}\in\mathbb{Z},\underline{m},\underline{n}\in\mathbb{Z}^{r}$,
and all $a_{i}(m_{0},\underline{m})$ act locally nilpotently on every
integrable $\tau$-module.
\end{lem}

\section{Category $\mathcal{E}_{\tau}$ of $\tau$-modules}
\def\theequation{3.\arabic{equation}}
\setcounter{equation}{0}

In this section, we define and study a category $\mathcal{E}_{\tau}$
of modules for the toroidal Lie algebra $\tau$, which is analogous to the
category $\mathcal{E}$ studied in \cite{L1} of modules for the affine Lie algebra $\hat{\g}$.
The category $\mathcal{E}_{\tau}$ is proved to
contain the evaluation modules studied by S. Eswara Rao. Moreover, it
is proved that every irreducible integrable $\tau$-module in
category $\mathcal{E}_{\tau}$ is isomorphic to a finite dimensional
evaluation module.

\begin{definition}\label{}
          Category $\mathcal{E}_{\tau}$ is defined to consist of $\tau$-modules
          $W$ satisfying the condition that there exist nonzero polynomials
          $p_{i}(x)\in\mathbb{C}[x]$, $i= 0,1,\dots,r$, such that
          $$p_{i}(x_{i})a(x_{0},\underline{x})w = 0 \;\; \mbox{for}\;\; i = 0,1,\dots,r,\ a\in
\mathfrak{g}, w\in W. $$
\end{definition}

\begin{lem}\label{lem:3.2}
The central elements $K_{0}(\underline{n})$ and $K_{i}$ of $\tau$ for $i=1,\dots,r,
            \underline{n}\in\mathbb{Z}^{r}$ act
              trivially on every $\tau$-module
               in category $\mathcal{E}_{\tau}.$
\end{lem}
\begin{proof}Let $W$ be a $\tau$-module in category $\mathcal{E}_{\tau}$
with nonzero polynomials $p_{i}(x)$ such that
$p_{i}(x_{i})a(x_{0},\underline{x})w=0$ for any $a\in \g, w\in W,
i=0,1,\dots,r.$ First, we consider the case that
$p_{i}(x)\in \mathbb{C}$ for some $i\in\{0,1,\dots,r\}$. In this case, we have
$a(x_{0},\underline{x})w=0$, i.e., $a(n_{0},\underline{n})w=0$ for
all $a\in \mathfrak{g}$, $w\in W$, $n_{0}\in\mathbb{Z}$,
$\underline{n}\in\mathbb{Z}^{r}.$ Then by using the Lie relations
(\ref{eq:2.2}), we have

\begin{align}
0=\left(n_{0}\langle a , b\rangle
\delta_{n_{0}+m_{0},0}K_{0}(\underline{n}+\underline{m})
   +\langle a , b\rangle \delta_{n_{0}+m_{0},0}
  \delta_{\underline{n}+\underline{m},0}\sum_{i=1}^{r}n_{i}K_{i}\right)w
  \label{eq:3.1}
\end{align}
 for all $b\in \mathfrak{g}, n_{0},m_{0}\in\mathbb{Z}$,
$\underline{n},\underline{m}\in\mathbb{Z}^{r}$, $w\in W$. Taking
$n_{0}=0$ in (\ref{eq:3.1}), we get
\begin{align}0=\langle a ,
b\rangle \delta_{m_{0},0}
\delta_{\underline{n}+\underline{m},0}\left(\sum_{i=1}^{r}n_{i}K_{i}\right)w.
\end{align}

For each fixed $i\in\{1,\dots,r\}$, by taking $n_{j}=\delta_{ij}$ for $j\in\{1,\dots,r\}$,
$m_{0}=0$ and $\underline{n}+\underline{m}=0$, we get
$K_{i}=0$ on $W$. Thus the identity (\ref{eq:3.1}) becomes
\begin{align}
0=n_{0}\langle a , b\rangle
\delta_{n_{0}+m_{0},0}K_{0}(\underline{n}+\underline{m})w.\label{eq:3.2}
\end{align}
Taking $n_{0}\neq 0$ and $n_{0}+m_{0}=0$ in
(\ref{eq:3.2}), we get
$K_{0}(\underline{n}+\underline{m})w=0$ for all
$\underline{n},\underline{m}\in\mathbb{Z}^{r}, w\in W.$

Now we consider the case that $p_{i}(x)$ are of positive degrees for all
$i=0,1,\dots,r.$ That is, $p'_{i}(x)\neq 0$ for all
$i=0,1,\dots,r.$ Pick $a,b\in \mathfrak{g}$ such that $\langle a, b
\rangle=1$. For $s=1,\dots,r$, by using the relations
(\ref{eq:2.3}), we have
 \begin{eqnarray}
         0 &=& p_{s}(x_{s})p_{s}(y_{s})[a(x_{0},\underline{x}),b(y_{0},\underline{y})]  \nonumber\\
            &=& p_{s}(x_{s})p_{s}(y_{s})K_{0}(\underline{y})(\frac{\partial}{\partial y_{0}}y_{0}^{-1}\delta(\frac{x_{0}}{y_{0}}))
                  \prod_{i=1}^{r}y_{i}^{-1}\delta(\frac{x_{i}}{y_{i}})\nonumber\\
         &&{}+p_{s}(x_{s})p_{s}(y_{s}) x_{0}^{-1}y_{0}^{-1}\delta(\frac{x_{0}}{y_{0}})
                 \sum_{j=1}^{r}(K_{j}(\frac{\partial}{\partial y_{j}}y_{j}^{-1}\delta(\frac{x_{j}}{y_{j}}))
                    \prod_{i=1,i\neq j}^{r}x_{i}^{-1}y_{i}^{-1}\delta(\frac{x_{i}}{y_{i}})). \label{eq:3.3}
            \end{eqnarray}

        Notice that Res$_{y_{0}}(\frac{\partial}{\partial y_{0}}y_{0}^{-1}\delta(\frac{x_{0}}{y_{0}}))=0$
        and Res$_{y_{i}}p_{i}(x_{i})p_{i}(y_{i})
                    (\frac{\partial}{\partial y_{i}}y_{i}^{-1}\delta(\frac{x_{i}}{y_{i}}))=
                    - p_{i}(x_{i})p'_{i}(x_{i})$ for $i=0,1,\dots,r.$
 Taking Res$_{y_{0}}$Res$_{y_{1}}$Res$_{y_{2}}$$\cdots$Res$_{y_{r}}$ of the identity
 (\ref{eq:3.3}), we obtain
            $$ K_{s}p_{s}(x_{s})p'_{s}(x_{s})x_{0}^{-1}x_{1}^{-1}\cdots x_{s-1}^{-1}x_{s+1}^{-1}\cdots x_{r}^{-1}=0.$$
            Thus we have $K_{s} = 0 $ on $W$ for $s=1,\dots,r$.
            Now multiplying the identity (\ref{eq:2.3}) by $p_{0}(x_{0})p_{0}(y_{0})$, we get
            \begin{eqnarray}
             \qquad\qquad\qquad\qquad 0 & = & p_{0}(x_{0})p_{0}(y_{0})[a(x_{0},\underline{x}),b(y_{0},\underline{y})]  \nonumber\\
              &  = & p_{0}(x_{0})p_{0}(y_{0})K_{0}(\underline{y})(\frac{\partial}{\partial y_{0}}y_{0}^{-1}\delta(\frac{x_{0}}{y_{0}}))
                  \prod_{i=1}^{r}y_{i}^{-1}\delta(\frac{x_{i}}{y_{i}}).           \label{eq:3.4}
            \end{eqnarray}
            Taking Res$_{y_{0}}$Res$_{x_{1}}$ $\cdots$Res$_{x_{r}}$ of the identity (\ref{eq:3.4})
            and using
            $y_{i}^{-1}\delta(\frac{x_{i}}{y_{i}})=x_{i}^{-1}\delta(\frac{y_{i}}{x_{i}})$, we get
            $$0=p_{0}(x_{0})p'_{0}(x_{0})K_{0}(\underline{y}).$$ Thus $K_{0}(\underline{n})=0$ on $W$ for all $\underline{n}\in\mathbb{Z}^{r}$.
\end{proof}

We now give some examples of $\tau$-modules in category
$\mathcal{E}_{\tau}$. Let $U$ be a $\mathfrak{g}$-module and let
$\underline{z} = (z_{0},\dots,z_{r})\in (\mathbb{C}^{\ast})^{r+1}$,
where $\mathbb{C}^{\ast}$ denotes the set of nonzero complex numbers. We
define an action of $\tau$ on $U$ by
 \begin{eqnarray}
 a(n_{0},\underline{n})u =
z_{0}^{n_{0}}\cdots z_{r}^{n_{r}}(au)   \label{eq:3.5}
\end{eqnarray}
 for $a\in \mathfrak{g},n_{0}\in
\mathbb{Z},\underline{n} = (n_{1},\dots,n_{r})\in\mathbb{Z}^{r},$
and by letting the whole center act trivially on $U$. Clearly, this makes $U$ a
$\tau$-module, which is denoted by $U(\underline{z}).$ More generally,
let $U_{1},\dots,U_{s}$ be $\mathfrak{g}$-modules and let
$\underline{z}_{j} =
(z_{0j},\dots,z_{rj})\in(\mathbb{C}^{*})^{r+1}$ for $j =
1,\dots,s.$ Then $U = U_{1}\otimes\cdots\otimes U_{s}$ is a $\tau$-
module where the whole center act trivially and
 \begin{eqnarray}
 a(n_{0},\underline{n})(u_{1}\otimes\cdots\otimes u_{s}) =
  \sum_{j = 1}^{s}z_{0j}^{n_{0}}\cdots z_{rj}^{n_{r}}
      (u_{1}\otimes\cdots\otimes au_{j}\otimes\cdots\otimes u_{s})  \label{eq:3.6}
       \end{eqnarray}
for $a\in
\mathfrak{g},n_{0}\in\mathbb{Z},\underline{n}\in\mathbb{Z}^{r},u_{j}\in
U_{j}.$ We denote this $\tau$-module by $\otimes_{j=1}^{s}U_{j}(\underline{z}_{j})$
and call it an {\em evaluation module.}

Next, we show that $\otimes_{j=1}^{s}U_{j}(\underline{z}_{j})$
belongs to the category $\mathcal{E}_{\tau}$. For $a\in
\mathfrak{g}$, $u_{j}\in U_{j}(\underline{z}_{j})$, we write
(\ref{eq:3.6}) in terms of generating functions to get
\begin{eqnarray*}
a(x_{0},\underline{x})(u_{1}\otimes\cdots\otimes u_{s})
&=&\sum_{n_{0}\in\mathbb{Z}}\sum_{\underline{n}\in\mathbb{Z}^{r}}
a(n_{0},\underline{n})x_{0}^{-n_{0}-1}x^{-\underline{n}-1}(u_{1}\otimes\cdots\otimes u_{s})\\
&=&\sum_{j = 1}^{s}x_{0}^{-1}\cdots x_{r}^{-1}
\delta(\frac{z_{0j}}{x_{0}})\cdots\delta(\frac{z_{rj}}{x_{r}})
      (u_{1}\otimes\cdots\otimes au_{j}\otimes\cdots\otimes u_{s}).
\end{eqnarray*}
Since $(x_{i}-z_{ij})\delta(\frac{z_{ij}}{x_{i}}) = 0 $ for $i =
0,1,\dots,r$, $j = 1,\dots,s,$ we have
$$ (x_{i}-z_{i1})\cdots (x_{i}-z_{is})
           a(x_{0},\underline{x})(u_{1}\otimes\cdots\otimes u_{s}) = 0 $$
for $i = 0,1,\dots,r.$

To summarize we have:

\begin{lem}\label{lem:3.3} Let $U_{1},\dots,U_{s}$ be $\mathfrak{g}$-modules and let $z_{ij}$
be nonzero complex numbers, $i = 0,1,\dots,r$, $j = 1,\dots,s.$
Set $p_{i}(x) = (x-z_{i1})\cdots (x-z_{is})$ for $i=
0,1,\dots,r$. Then the tensor product $\tau$-module
$\otimes_{j=1}^{s}U_{j}(\underline{z}_{j})$ is in the category
$\mathcal{E}_{\tau}$, where $\underline{z}_{j} =
(z_{0j},\dots,z_{rj})\in(\mathbb{C}^{*})^{r+1}.$
\end{lem}

 In the following we shall classify irreducible integrable $\tau$-modules in category
$\mathcal{E}_{\tau}$. For $a\in \mathfrak{g}$, we have

$$a(x_{0},\underline{x}) =
\sum_{n_{0}\in\mathbb{Z}}
\sum_{\underline{n}\in\mathbb{Z}^{r}}
(a\otimes t_{0}^{n_{0}}t^{\underline{n}})x_{0}^{-n_{0}-1}x^{-\underline{n}-1}
 = a\otimes \left(\prod_{i=0}^{r}x_{i}^{-1}\delta(\frac{t_{i}}{x_{i}})\right).  $$

For each $f_{k}(x)\in\mathbb{C}[x],k\in\{0,1,\dots,r\},m_{0}\in\mathbb{Z},
\underline{m}=(m_{1},\dots,m_{r})\in\mathbb{Z}^{r},a\in \mathfrak{g}$,
we have
\begin{eqnarray*}
\qquad\qquad\quad(\prod_{j=0}^{r}x_{j}^{m_{j}})f_{k}(x_{k})a(x_{0},\underline{x})&
= & a\otimes (\prod_{j=0}^{r}x_{j}^{m_{j}})f_{k}(x_{k})(\prod_{i=0}^{r}x_{i}^{-1}\delta(\frac{t_{i}}{x_{i}}))\\
&= & a\otimes (\prod_{j=0}^{r}t_{j}^{m_{j}})f_{k}(t_{k})
(\prod_{i=0}^{r}x_{i}^{-1}\delta(\frac{t_{i}}{x_{i}})).
\qquad\qquad\quad
\end{eqnarray*}

Since Res$_{x_{r}}\cdots$ Res$_{x_{1}}$Res$_{x_{0}}
(\prod\limits_{j=0}^{r}x_{j}^{m_{j}})f_{k}(x_{k})a(x_{0},\underline{x})
=a\otimes (\prod\limits_{j=0}^{r}t_{j}^{m_{j}})f_{k}(t_{k})$, we
have that for any $\tau$-module $W$, if
$f_{k}(x_{k})a(x_{0},\underline{x})W=0$, then $(a\otimes
f_{k}(t_{k})\mathbb{C}[t_{0}^{\pm 1},t_{1}^{\pm 1},\dots,t_{r}^{\pm
1}])W=0.$

For nonzero polynomials $p_{i}(x)$, let $W$ be a $\tau$-module such that
$$p_{i}(x_{i})a(x_{0},\underline{x})w = 0$$
for $a\in \mathfrak{g}, w\in W,i = 0,\dots,r.$ We denote by
$P=\langle p_{0}({t_{0}}),\dots,p_{r}(t_{r})\rangle$ the ideal of
$\mathbb{C}[t_{0}^{\pm 1},t_{1}^{\pm 1},\dots,t_{r}^{\pm 1}]$
generated by $p_{i}({t_{i}}), i = 0,\dots,r.$ Then $W$ is a module
for the Lie algebra  $\g\otimes\mathbb{C}[t_{0}^{\pm 1},t_{1}^{\pm
1},\dots,t_{r}^{\pm 1}]/P$. We have the following result.


\begin{prop}\label{pro:3.4} Any finite dimensional irreducible $\tau$-module
$W$ from the category $\mathcal{E}_{\tau}$ is isomorphic to a
$\tau$-module $U_{1}(\underline{z}_{1})\otimes\cdots\otimes
U_{s}(\underline{z}_{s})$ for some finite dimensional irreducible
$\mathfrak{g}$-modules $U_{1},\dots,U_{s}$, and for $s$ distinct
$(r+1)$-tuples $\underline{z}_{i} =
(z_{0i},z_{1i},\dots,z_{ri})\in(\mathbb{C}^{*})^{r+1}$
($i=1,\ldots,s$).

\end{prop}
\begin{proof} Let $p_{i}(x)$ be nonzero polynomials such that
              $p_{i}(x_{i})a(x_{0},\underline{x})w=0$
              for $i=0,1,\dots,r$, $w\in W$.
              Assume that
              $z_{i,1},\dots,z_{i,N_{i}}\in\mathbb{C}^{*}$
              are distinct nonzero roots of $p_{i}(x_{i})$, where $N_{i}$ is the number of
distinct nonzero roots of the polynomial $p_{i}(x_{i})$ for $0\leq
i\leq r$. Let $s=N_{0}\cdots N_{r}$.
              Set $\tau_{r}=\mathfrak{g}\otimes\mathbb{C}[t_{0}^{\pm
              1}, t_{1}^{\pm 1},\dots,t_{r-1}^{\pm 1}]$,
              $\tau_{r-1}=\mathfrak{g}\otimes\mathbb{C}[t_{0}^{\pm
              1}, t_{1}^{\pm 1},\dots, t_{r-2}^{\pm 1}]$,
              $\dots,$
              $\tau_{2}=\mathfrak{g}\otimes\mathbb{C}[t_{0}^{\pm 1}, t_{1}^{\pm 1}]$,
              $\tau_{1}=\mathfrak{g}\otimes\mathbb{C}[t_{0}^{\pm
              1}]$. Let
              $\widehat{\tau_{i}}=\tau_{i}\otimes{\mathbb{C}}[t_{i}]$
              for $1\leq i\leq r$.
              By Lemma \ref{lem:3.2}, the while center of $\tau$ act trivially on $W$, thus $W$ can be viewed as a
              $\widehat{\tau_{r}}$-module
              with the property
              $p_{r}(x_{r})a(x_{0},\underline{x})w=0$ for $a\in \mathfrak{g}, w\in
              W$,
             and $[\tau_{r},\tau_{r}]w
              =\tau_{r}w$ for $w\in W$.
              From the Proposition 3.9 of \cite{L1},
              we have that
              $W$ is isomorphic to a $\tau$-module
              $U^{(r)}_{1}(z_{r,1})\otimes\cdots\otimes U^{(r)}_{N_{r}}(z_{r,N_{r}})$
              for some finite dimensional irreducible
              $\tau_{r}$-modules
              $U^{(r)}_{1},\dots,U^{(r)}_{N_{r}}$, and distinct nonzero complex numbers $z_{r,1},\dots,z_{r,N_{r}}$.
              Since $p_{r-1}(x_{r-1})a(x_{0},\underline{x})w=0$,
              i.e., $p_{r-1}(x_{r-1})a(x_{0},\underline{x})u_{1}\otimes\cdots\otimes u_{N_{r}}=0$
              for $u_{i}\in U^{(r)}_{i}$, $i=1,\dots,N_{r}$.
              It follows that
              $$\sum\limits_{j=1}^{N_{r}}p_{r-1}(x_{r-1})z_{r,j}^{n_{r}}
              (u_{1}\otimes\cdots\otimes a(x_{0},x_{1},\dots,x_{r-1})u_{j}\otimes\cdots\otimes
              u_{N_{r}})=0$$
              for any $n_{r}\in\mathbb{Z}$, where $$a(x_{0},x_{1},\dots,x_{r-1})=\sum\limits_{{n_{0},\dots,n_{r-1}\in
              {\mathbb{Z}}}}(a\otimes t_{0}^{n_{0}}\cdots t_{r-1}^{n_{r-1}})x_{0}^{-n_{0}-1}\cdots
              x_{r-1}^{-n_{r-1}-1}.$$
Note that $z_{r,j}$ are distinct for $1\leq j\leq N_{r}$. By taking
$n_{r}= 0,\ldots, N_{r}-1$ and then
               using the Vandermonde determinant property, we get that
              $p_{r-1}(x_{r-1})a(x_{0},x_{1},\dots,x_{r-1})U_{j}^{(r)}=0$
              for $1\leq j\leq N_{r}$.
              Thus for every $1\leq j\leq N_{r}$,
              $U_{j}^{(r)}$ is a finite dimensional $\tau_{r}$-module with
               $p_{r-1}(x_{r-1})a(x_{0},x_{1},\dots,x_{r-1})U_{j}^{(r)}=0$.
               Since $[\tau_{r-1},\tau_{r-1}]=\tau_{r-1}$ and $\tau_{r}=\widehat{\tau_{r-1}}$,
                by using the Proposition 3.9 of \cite{L1} again, we have that
               $U_{j}^{(r)}$ is isomorphic to $U_{j,1}^{(r-1)}(z_{r-1,1})\otimes\cdots\otimes U^{(r-1)}_{j,N_{r-1}}(z_{r-1,N_{r-1}})$
               for some finite dimensional $\tau_{r-1}$-module
               $U^{(r-1)}_{j,1},\dots,U^{(r-1)}_{j,N_{r-1}}$.
               Therefore $W$ is isomorphic to a module of the from
               $$U_{1,1}^{(r-1)}(z_{r,1},z_{r-1,1})\otimes\cdots\otimes U^{(r-1)}_{1,N_{r-1}}(z_{r,1},z_{r-1,N_{r-1}})
              \otimes\cdots\otimes
              U^{(r-1)}_{N_{r},N_{r-1}}(z_{r,N_{r}},z_{r-1,N_{r-1}})$$
              for some finite dimensional irreducible
              $\tau_{r-1}$-modules
              $U^{(r-1)}_{1,1},\dots,U^{(r-1)}_{1,N_{r-1}},\dots,U^{(r-1)}_{N_{r},N_{r-1}}$.
               Then the proposition follows from recursion.
\end{proof}

\begin{prop}\label{pro:3.5}
Irreducible integrable $\tau$-modules in the category
$\mathcal{E}_{\tau}$ up to isomorphism are exactly the evaluation
modules $U_{1}(\underline{z}_{1})\otimes\cdots\otimes
U_{s}(\underline{z}_{s})$, where  $U_{1},\dots,U_{s}$ are finite
dimensional irreducible $\mathfrak{g}$-modules and
$\underline{z}_{i} =
(z_{0i},z_{1i},\dots,z_{ri})\in(\mathbb{C}^{*})^{r+1}$  for
$i=1,\ldots,s$ are distinct $(r+1)$-tuples.
\end{prop}

\begin{proof}Let $W$ be an irreducible integrable $\tau$-module in the category
          $\mathcal{E}_{\tau}$. Then there exist nonzero
          polynomials $p_{i}(x)$ such that
          $a\otimes p_{i}(t_{i})\mathbb{C}[t_{0}^{\pm},\dots,t_{r}^{\pm}]W=0$
          for $a\in \mathfrak{g}, i = 0,1,\dots,r.$ Let $I$ be the annihilating
          ideal of the $\tau$-module $W$. Since
          $a\otimes \langle p_{0}(t_{0}),\cdots, p_{r}(t_{r})\rangle\subset I$ for all $a\in \mathfrak{g}$,
it follows that $\tau/I$ is finite dimensional.
          By Lemma \ref{lem:2.3}, there is a basis $\{a_{1},\dots,a_{l}\}$ of $\mathfrak{g}$ such that
          $a_{i}(m_{0},\underline{m})$ acts locally nilpotently on $W$
          for any $i\in\{1,\dots,l\}$.
          Let $0\neq w\in W$. We have $W=U(\tau)w=U(\tau/I)w$ by the irreducibility of $W$.
          From the PBW theorem,  $W$ is finite dimensional. It then follows from Proposition \ref{pro:3.4}.
\end{proof}

\section{Category $\mathcal{C}_{\tau}$ of $\tau$-modules}\label{sec 4}
\def\theequation{4.\arabic{equation}}
\setcounter{equation}{0}

In this section, we define and study a category $\mathcal{C}_{\tau}$
of modules for the toroidal Lie algebra $\tau$, we also introduce a
category $\mathcal{\widetilde{R}}$ of $\tau$-modules and a category
$\mathcal{E}_{\tau}^{'}$ of $\tau$-modules, where
$\mathcal{\widetilde{R}}$ is a subcategory of the category of
restricted $\tau$-modules and $\mathcal{E}_{\tau}^{'}$ is a
subcategory of the category $\mathcal{E}_{\tau}$. We prove that
every irreducible $\tau$-module in category $\mathcal{C}_{\tau}$ is
isomorphic to the tensor product of modules from categories
$\mathcal{\widetilde{R}}$ and $\mathcal{E}_{\tau}^{'}$.

\begin{definition}\label{t1} We define category
                $\mathcal{C}_{\tau}$ to consist of $\tau$-modules $W$ for which there exist nonzero polynomials
                $p_{i}(x)$ for
                $0\le i\le r$ such that that the nonzero roots of $p_{i}(x)$ for $1\le i\le r$ are multiplicity-free,
                $$ p_{0}(x_{0})a(x_{0},\underline{x})\in
                {Hom}(W,W[[x_{1}^{\pm 1},\dots,x_{r}^{\pm 1}]]((x_{0})))
                $$ for $a\in \mathfrak{g}$,
                and $$ p_{i}(x_{i})a(x_{0},\underline{x})w = 0, \;\; p_{i}(x_{i})K_{0}(\underline{x})w=0
                $$
for $1\le i\le r$, $w\in W.$
\end{definition}

\begin{definition}\label{} We define category
                $\mathcal{\widetilde{R}}$  to consist of $\tau$-modules $W$
                for which there exist nonzero polynomials $p_{i}(x)$  for $1\le i\le r$
                such that the nonzero roots of
$p_{i}(x)$  are multiplicity-free,
                $$ a(x_{0},\underline{x})\in
                {Hom}(W,W[[x_{1}^{\pm},\dots,x_{r}^{\pm}]]((x_{0})))
                $$ for $a\in \mathfrak{g}$,
                and $$ p_{i}(x_{i})a(x_{0},\underline{x})w = 0,\;\;
                 p_{i}(x_{i})K_{0}(\underline{x})w =0$$ for $1\le i\le r$, $w\in W.$
\end{definition}

\begin{definition}\label{}
          We define a category $\mathcal{E}^{'}_{\tau}$ to consist of $\tau$-modules
          $W$ for which there exist nonzero polynomials $p_{i}(x)\in\mathbb{C}[x]$
          for $0\le i\le r$ such that the
      nonzero roots of $p_{i}(x)$ for $1\le i\le r$ are multiplicity-free and
          $$ p_{i}(x_{i})a(x_{0},\underline{x})w = 0 $$
          for $0\le i\le r, a\in \mathfrak{g}, w\in W.$
          \end{definition}

          \begin{rmk}\label{rmk:4.1} From the proof of Proposition
          3.9 in \cite{L1} and Proposition \ref{pro:3.5}, we see
          that irreducible integrable $\tau$-modules in category
          $\mathcal{E}^{'}_{\tau}$ are exactly those irreducible integrable $\tau$-modules in category
          $\mathcal{E}_{\tau}$, i.e., irreducible integrable $\tau$-modules in category
$\mathcal{E}_{\tau}^{'}$ up to isomorphism are the evaluation
modules $U_{1}(\underline{z}_{1})\otimes\cdots\otimes
U_{s}(\underline{z}_{s})$, where  $U_{1},\dots,U_{s}$ are finite
dimensional irreducible $\mathfrak{g}$-modules and
$\underline{z}_{i} =
(z_{0i},z_{1i},\dots,z_{ri})\in(\mathbb{C}^{*})^{r+1}$  for $1\le
i\le s$ are distinct $(r+1)$-tuples. \end{rmk}

It is obvious that every $\tau$-module in category
$\mathcal{\widetilde{R}}$ and every $\tau$-module in category
$\mathcal{E}_{\tau}^{'}$ are in $\mathcal{C}_{\tau}$ and the tensor
products of $\tau$-modules from $\mathcal{\widetilde{R}}$
and $\mathcal{E}_{\tau}^{'}$ are in $\mathcal{C}_{\tau}.$


From the proof of Lemma \ref{lem:3.2}, we obtain

\begin{lem}The central elements $K_{i}$ for $1\le i\le r$ act trivially on
        every module from categories $\mathcal{C}_{\tau}$ and
         $\mathcal{\widetilde{R}}$.
\end{lem}

Next, we define some vector spaces of formal series which will be used
later.

\begin{definition}\label{def:4.4}
Let $W$ be any vector space. Following \cite{LTW} we set
$$\mathcal{E}(W,r)= Hom(W,W[[x_{1}^{\pm 1},\dots,x_{r}^{\pm
 1}]]((x_{0}))).$$
 Define $\overline{\mathcal{E}}(W,r)$ to be the subspace of
            (End$W$)$[[x_{0}^{\pm 1},\dots,x_{r}^{\pm 1}]]$,
            consisting of formal series $\alpha(x_{0},\underline{x})$ satisfying the condition that
            there exist nonzero
            polynomials $p_{0}(x),\dots,p_{r}(x)$
            such that  the
            nonzero roots of $p_{i}(x)$ for $1\le i\le r$ are
            multiplicity-free,
            $$p_{0}(x_{0})\alpha(x_{0},\underline{x})\in\mathcal{E}(W,r),\ \mbox{ and }\
            p_{i}(x_{i})\alpha (x_{0},\underline{x})w = 0$$
            for $1\le i\le r,\ w\in W$.
            Define  $\widetilde{\mathcal{E}}(W,r)$ to be the subspace of
            (End$W$)$[[x_{0}^{\pm 1},\dots,x_{r}^{\pm 1}]]$,
            consisting of formal series $\alpha (x_{0},\underline{x})$
            satisfying the condition that  there exist
            nonzero polynomials $p_{1}(x),\dots,p_{r}(x)$ whose nonzero roots are
            multiplicity-free such that
            $$\alpha(x_{0},\underline{x})\in\mathcal{E}(W,r),\ \
           p_{i}(x_{i})\alpha(x_{0},\underline{x})w =0$$
           for $1\le i\le r,\ w\in W$.
            Furthermore, define $\overline{\mathcal{E}_{0}}(W,r)$ to be the subspace
            of $\overline{\mathcal{E}}(W,r)$ consisting of
            formal series $\alpha (x_{0},\underline{x})$
            satisfying the condition that there exist nonzero
            polynomials $p_{0}(x),\dots,p_{r}(x)$
            such that the nonzero roots of $p_{i}(x)$  for $1\le i\le r$ are
            multiplicity-free and
            $p_{i}(x_{i})\alpha(x_{0},\underline{x})w = 0 $
            for $0\le i\le r$, $w\in W$.
\end{definition}

\begin{definition}\label{def:4.5} For a vector space $W$, we define a linear map
             $$ \psi_{\widetilde{\mathcal{R}}}: \overline{\mathcal{E}}(W,r)
             \rightarrow \widetilde{\mathcal{E}}(W,r)$$
             by

\begin{equation}\label{eq:4.1} \psi_{\widetilde{\mathcal{R}}}(\alpha(x_{0},\underline{x}))w
            = l_{x_{0};0}(f(x_{0})^{-1})(f(x_{0})\alpha(x_{0},\underline{x})w), \;\;
\end{equation}
for $\alpha(x_{0},\underline{x})\in\overline{\mathcal{E}}(W,r),w\in W$,
           where $f(x)$ is any nonzero polynomial such that\\
            $f(x_{0})\alpha(x_{0},\underline{x})\in\mathcal{E}(W,r).$
\end{definition}

Notice that
\begin{eqnarray}
(End W)[[x_{1}^{\pm 1},\dots,x_{r}^{\pm
 1}]]\subset \mathcal{E}(W,r)\subset \overline{\mathcal{E}}(W,r).
 \end{eqnarray}
Thus for any $\beta(\underline{x})\in (End W)[[x_{1}^{\pm 1},\dots,x_{r}^{\pm
 1}]]$, we have
$$\psi_{\widetilde{\mathcal{R}}}(\beta(\underline{x}))=\beta(\underline{x}).$$

\begin{rmk}\label{rmk:4.2}  Just as in \cite{L1}, one can show that $\psi_{\widetilde{\mathcal{R}}}$ is well defined. From definition, for $\alpha(x_{0},\underline{x})\in\overline{\mathcal{E}}(W,r)$,
        we have
        $ f_{0}(x_{0})\psi_{\widetilde{\mathcal{R}}}(\alpha(x_{0},\underline{x}))
         = f_{0}(x_{0})\alpha(x_{0},\underline{x}) $ for
           any nonzero polynomial $f_{0}(x_{0})$ such that
            $f_{0}(x_{0})\alpha(x_{0},\underline{x})\in\mathcal{E}(W,r).$
            Furthermore, if $f_{i}(x_{i})$ for $1\le i\le r$ are nonzero
polynomials such that  $f_{i}(x_{i})\alpha(x_{0},\underline{x})=0$,
then $f_{i}(x_{i})\psi_{\widetilde{R}}(\alpha(x_{0},\underline{x}))=0$
 for
$1\le i\le r.$
\end{rmk}

\begin{prop}\label{prop:4.8}For any vector space $W$, we have
     \begin{equation}\label{eq:4.2} \overline{\mathcal{E}}(W,r) = \widetilde{\mathcal{E}}(W,r)
            \oplus \overline{\mathcal{E}_{0}}(W,r)
             \end{equation}
     Furthermore,
     $$\psi_{\widetilde{\mathcal{R}}}|_{\widetilde{\mathcal{E}}(W,r)} = 1\ \mbox{ and }\
        \psi_{\widetilde{\mathcal{R}}}|_{\overline{\mathcal{E}_{0}}(W,r)} = 0.$$
 \end{prop}
 \begin{proof} Let $\alpha(x_{0},\underline{x})\in\widetilde{\mathcal{E}}(W,r)$.
             By Definition \ref{def:4.4}, we know $\alpha(x_{0},\underline{x})\in\mathcal{E}(W,r)$.
              Taking $f_{0}(x)=1$ in Definition \ref{def:4.5},
             we have $\psi_{\widetilde{\mathcal{R}}}(\alpha(x_{0},\underline{x}))w
                            = \alpha(x_{0},\underline{x})w$ for any $w\in
                            W$,
             that is, $\psi_{\widetilde{\mathcal{R}}}(\alpha(x_{0},\underline{x}))
                            = \alpha(x_{0},\underline{x})$.
                            Therefore $\psi_{\widetilde{\mathcal{R}}}|_{\widetilde{\mathcal{E}}(W,r)} = 1$.

             Let $\alpha(x_{0},\underline{x})\in\overline{\mathcal{E}_{0}}(W,r)$. By definition there
             exist nonzero polynomials $p_{i}(x)$ for $0\le i\le r$
             such that the
            nonzero roots of $p_{i}(x)$ for $1\le i\le r$ are
            multiplicity-free  and
             $p_{i}(x_{i})\alpha(x_{0},\underline{x})w = 0 $ for $w\in W$.
            Then by (\ref{eq:4.1}) in Definition \ref{def:4.5}, we have
            $$\psi_{\widetilde{\mathcal{R}}}(\alpha(x_{0},\underline{x}))w
            = l_{x_{0};0}(p_{0}(x_{0})^{-1})(p_{0}(x_{0})\alpha(x_{0},\underline{x})w)=0.$$
            Thus
            $\psi_{\widetilde{\mathcal{R}}}(\alpha(x_{0},\underline{x}))=0$,
            i.e., $\psi_{\widetilde{\mathcal{R}}}|_{\overline{\mathcal{E}_{0}}(W,r)} =
            0.$ It is clear that
            $\widetilde{\mathcal{E}}(W,r)+\overline{\mathcal{E}_{0}}(W,r)$
            is a direct sum.

            Let $\alpha(x_{0},\underline{x})\in\overline{\mathcal{E}}(W,r)$
            and let $f_{i}(x)$ for $0\le i\le r$ be nonzero polynomials such that
            the nonzero roots of $f_{i}(x)$ for $1\le i\le r$ are
            multiplicity-free,
            $f_{0}(x_{0})\alpha(x_{0},\underline{x})\in\mathcal{E}(W,r)$,
            $f_{i}(x_{i})\alpha(x_{0},\underline{x})=0$ for $1\le i\le r$.
            From Remark \ref{rmk:4.2}, we have
            $f_{i}(x_{i})(\alpha(x_{0},\underline{x})-
                \psi_{\widetilde{\mathcal{R}}}(\alpha(x_{0},\underline{x})))=0$
            for $0\le i\le r$. Thus
             $\alpha(x_{0},\underline{x})-\psi_{\widetilde{\mathcal{R}}}(\alpha(x_{0},\underline{x}))
                  \in\overline{\mathcal{E}_{0}}(W,r)$,
                  which implies $\alpha(x_{0},\underline{x})\in\widetilde{\mathcal{E}}(W,r)\oplus
                                         \overline{\mathcal{E}_{0}}(W,r)$.
This proves $\overline{\mathcal{E}}(W,r) \subset
\widetilde{\mathcal{E}}(W,r) + \overline{\mathcal{E}_{0}}(W,r)$.
                          Therefore, we have (\ref{eq:4.2}),
                          and the proof is completed.
 \end{proof}

In what follows, we denote by $\psi_{\mathcal{E}^{'}}$ the
            projection map of $\overline{\mathcal{E}}(W,r)$ onto $\overline{\mathcal{E}_{0}}(W,r)$
            with respect to the decomposition (\ref{eq:4.2}). For $\alpha(x_{0},\underline{x})
            \in \overline{\mathcal{E}}(W,r)$, we set
            $$\widetilde{\alpha}(x_{0},\underline{x})=
            \psi_{\widetilde{\mathcal{R}}}(\alpha(x_{0},\underline{x})),$$
           $$ \check{\alpha}(x_{0},\underline{x})=\psi_{\mathcal{E}^{'}}(\alpha(x_{0},\underline{x}))
           =\alpha(x_{0},\underline{x})-\widetilde{\alpha}(x_{0},\underline{x}).$$

To prove our main result we need the following lemmas:

\begin{lem}\label{lem:4.10} For $\alpha(x_{0},\underline{x})\in \overline{\mathcal{E}}(W,r),
        n_{0}\in\mathbb{Z},\underline{n}\in\mathbb{Z}^{r},w\in W$,
        we have
        $$\psi_{\widetilde{\mathcal{R}}}(\alpha(x_{0},\underline{x}))(n_{0},\underline{n})w
        =\sum_{i=0}^{l}\beta_{i}\alpha(n_{0}+i,\underline{n})w  $$
        for some $l\in\mathbb{N},\beta_{1},\dots,\beta_{l}\in\mathbb{C}$,
        depending on $\alpha(x_{0},\underline{x}),w$ and $n_{0}$, $\underline{n}$,
        where $$\psi_{\widetilde{\mathcal{R}}}(\alpha(x_{0},\underline{x}))
        =\sum_{n_{0}\in\mathbb{Z},\underline{n}\in\mathbb{Z}^{r}}\psi_{\widetilde{\mathcal{R}}}(\alpha(x_{0},\underline{x}))
        (n_{0},\underline{n})x_{0}^{-n_{0}-1}x^{-\underline{n}-1}$$
\end{lem}
\begin{proof} Since $\alpha(x_{0},\underline{x})\in \overline{\mathcal{E}}(W,r)$, there exists a polynomial  $p_{0}(x)$ such that $p_{0}(0)\neq 0$
              and $p_{0}(x_{0})\alpha(x_{0},\underline{x})\in\mathcal{E}(W,r)$.
              Let $k$ be a nonnegative integer such that
              $$x_{0}^{k}p_{0}(x_{0})\alpha(x_{0},\underline{x})w\in W[[x_{1}^{\pm 1},\dots,x_{r}^{\pm
              1}]][[x_{0}]].$$
             Set
              $$l_{x_{0};0}\left(\frac{1}{p_{0}(x_{0})}\right)=\sum_{i\geq 0}\alpha_{i}x_{0}^{i}\in\mathbf{C}[[x_{0}]]. $$
              Note that
              Res$_{x_{0}}x_{0}^{k+m}p_{0}(x_{0})\alpha(x_{0},\underline{x})w=0$ for $m\geq 0$.
              We have
              \begin{eqnarray*}
              \qquad\qquad \sum_{\underline{n}\in\mathbb{Z}^{r}}\psi_{\widetilde{\mathcal{R}}}
                          (\alpha(x_{0},\underline{x}))(n_{0},\underline{n})x^{-\underline{n}-1}w
               & = & \mbox{Res}_{x_{0}}x_{0}^{n_{0}}\psi_{\widetilde{\mathcal{R}}}(\alpha(x_{0},\underline{x}))w     \\
               & = & \mbox{Res}_{x_{0}}\sum_{0\leq i\leq k-n_{0}-1}
                       \alpha_{i}x_{0}^{n_{0}+i}(p_{0}(x_{0})\alpha(x_{0},\underline{x})w)\\
               & = & \sum_{\underline{n}\in\mathbb{Z}^{r}}
                 \sum_{i=0}^{l}\beta_{i}\alpha(n_{0}+i,\underline{n})x^{-\underline{n}-1}w.   \qquad\qquad\:\qquad\qquad
               \end{eqnarray*}
             Then it follows immediately that
             $\psi_{\widetilde{\mathcal{R}}}(\alpha(x_{0},\underline{x}))(n_{0},\underline{n})
             =\sum\limits_{i=0}^{l}\beta_{i}\alpha(n_{0}+i,\underline{n})$.
\end{proof}

The following delta function properties can be found in \cite{L2}, \cite{LL}:
For $m>n\geq 0$, we have
 \begin{equation}\label{eq:4.3}
 (x_{1}-x_{2})^{m}\left(\frac{\partial}{\partial x_{2}}\right)^{n}
   x_{2}^{-1}\delta\left(\frac{x_{1}}{x_{2}}\right)=0,
 \end{equation}
 and for $0\leq m\leq n$, we have
   \begin{equation}\label{eq:4.4}(x_{1}-x_{2})^{m}\frac{1}{n!}\left(\frac{\partial}{\partial x_{2}}\right)^{n}x_{2}^{-1}\delta\left(\frac{x_{1}}{x_{2}}\right)=
   \frac{1}{(n-m)!}\left(\frac{\partial}{\partial x_{2}}\right)^{n-m}
   x_{2}^{-1}\delta\left(\frac{x_{1}}{x_{2}}\right).
   \end{equation}

\begin{lem}\label{lem:4.11} Let $W$ be any vector space, let
          $\alpha(x_{0},\underline{x}),\beta(x_{0},\underline{x})\in\overline{\mathcal{E}}(W,r)$, and let\\
          $\gamma_{0}(x_{0},\underline{x}),\dots,\gamma_{n}(x_{0},\underline{x})$
          be formal series in (End $W$)$[[x_{0}^{\pm 1},\dots,x_{r}^{\pm 1}]]$ such that on $W$,
         \begin{equation}\label{eq:4.5}
         [\alpha(x_{0},\underline{x}),\beta(y_{0},\underline{y})]=
           \sum_{j=0}^{n}\frac{1}{j!}\gamma_{j}(y_{0},\underline{y})
           \left(\frac{\partial}{\partial y_{0}}\right)^{j}x_{0}^{-1}
           \delta\left(\frac{y_{0}}{x_{0}}\right)
           \prod_{i=1}^{r}x_{i}^{-1}\delta\left(\frac{y_{i}}{x_{i}}\right).
          \end{equation}

           Then
           $\gamma_{0}(x_{0},\underline{x}),\dots,\gamma_{n}(x_{0},\underline{x})\in\overline{\mathcal{E}}(W,r)$,
           and
           \begin{equation}\label{eq:4.6}
           [\widetilde{\alpha}(x_{0},\underline{x}),\widetilde{\beta}(y_{0},\underline{y})]
           =\sum_{j=0}^{n}\frac{1}{j!}\widetilde{\gamma_{j}}(y_{0},\underline{y})
           \left(\frac{\partial}{\partial y_{0}}\right)^{j}x_{0}^{-1}\delta\left(\frac{y_{0}}{x_{0}}\right)
           \prod_{i=1}^{r}x_{i}^{-1}\delta\left(\frac{y_{i}}{x_{i}}\right).
           \end{equation}
\end{lem}

\begin{proof} First we note that
       \begin{eqnarray*}
       \qquad\qquad\qquad 0 & = & \mbox{Res}_{x_{0}}\gamma_{j}(y_{0},\underline{y})
       \left(\frac{\partial}{\partial y_{0}}\right)^{n}x_{0}^{-1}\delta\left(\frac{y_{0}}{x_{0}}\right)
            \prod_{i=1}^{r}x_{i}^{-1}\delta\left(\frac{y_{i}}{x_{i}}\right)\\
       & = & (-1)^{n}\mbox{Res}_{x_{0}}\gamma_{j}(y_{0},\underline{y})\left(\frac{\partial}{\partial x_{0}}\right)^{n}x_{0}^{-1}\delta\left(\frac{y_{0}}{x_{0}}\right)
           \prod_{i=1}^{r}x_{i}^{-1}\delta\left(\frac{y_{i}}{x_{i}}\right)  \qquad\qquad\qquad
       \end{eqnarray*}
       for $n\geq 1$.
       Then from (\ref{eq:4.3}), (\ref{eq:4.4}), (\ref{eq:4.5}), we obtain
$$\gamma_{j}(y_{0},\underline{y})\prod_{i=1}^{r}x_{i}^{-1}\delta\left(\frac{y_{i}}{x_{i}}\right)
       =\mbox{Res}_{x_{0}}(x_{0}-y_{0})^{j}[\alpha(x_{0},\underline{x}),\beta(y_{0},\underline{y})]$$
       Thus
       $$\gamma_{j}(y_{0},\underline{y})
       =\mbox{Res}_{x_{r}}\cdots \mbox{Res}_{x_{1}}\mbox{Res}_{x_{0}}(x_{0}-y_{0})^{j}
        [\alpha(x_{0},\underline{x}),\beta(y_{0},\underline{y})] $$
        for $0\leq j\leq n$.
       Since $\beta(x_{0},\underline{x})\in\overline{\mathcal{E}}(W,r)$, we see obviously that
        $\gamma_{j}(x_{0},\underline{x})\in\overline{\mathcal{E}}(W,r)$
       for $j=0,\dots,n$.
 Let $0\neq f_{0}(x_{0})\in\mathbb{C}[x_{0}]$ be such that
       $$f_{0}(x_{0})\alpha(x_{0},\underline{x}), f_{0}(x_{0})\beta(x_{0},\underline{x}),
       f_{0}(x_{0})\gamma_{j}(x_{0},\underline{x})\in\mathcal{E}(W,r)$$
       for $j=0,\dots,n.$ So we have
       $$f_{0}(x_{0})\alpha(x_{0},\underline{x})=f_{0}(x_{0})\widetilde{\alpha}(x_{0},\underline{x}),
        \ \ \ \ f_{0}(x_{0})\beta(x_{0},\underline{x})=f_{0}(x_{0})\widetilde{\beta}(x_{0},\underline{x})$$
         $$f_{0}(x_{0})\gamma_{j}(x_{0},\underline{x})=f_{0}(x_{0})\widetilde{\gamma_{j}}(x_{0},\underline{x})$$
       for $j=0,\dots,n.$

       Multiplying both sides of (\ref{eq:4.5}) by $f_{0}(x_{0})f_{0}(y_{0})$
       we obtain
       $$f_{0}(x_{0})f_{0}(y_{0})
         [\widetilde{\alpha}(x_{0},\underline{x}),\widetilde{\beta}(y_{0},\underline{y})]
         = \sum_{j=0}^{n}\frac{1}{i!}f_{0}(x_{0})f_{0}(y_{0})
            \widetilde{\gamma_{j}}(y_{0},\underline{y})
           (\frac{\partial}{\partial y_{0}})^{j}x_{0}^{-1}\delta(\frac{y_{0}}{x_{0}})
           \prod_{i=1}^{r}x_{i}^{-1}\delta(\frac{y_{i}}{x_{i}}).$$
         Then we multiply both sides by
         $l_{x_{0};0}(f_{0}(x_{0})^{-1})l_{y_{0};0}(f_{0}(y_{0})^{-1})$
         to get (\ref{eq:4.6}).
\end{proof}

\begin{thm}\label{thm4.8} Let $\pi:
           \tau\rightarrow $ End $W$ be a representation of toroidal Lie algebra $\tau$ in
           category $\mathcal{C}_{\tau}$. Define linear maps $\pi_{\widetilde{R}}$
           and $\pi_{\mathcal{E}^{'}}$ from $\tau$ to End($W$) in terms
           of generating functions by
          \begin{eqnarray*}\label{eq:4.8}
           \pi_{\widetilde{\mathcal{R}}}
           \left(a(x_{0},\underline{x})+\alpha_{0}K_{0}(\underline{x})+\sum_{i=1}^{r}\alpha_{i}K_{i}\right)
             =
             \psi_{\widetilde{\mathcal{R}}}(\pi(a(x_{0},\underline{x}))+\alpha_{0}\pi(K_{0}(\underline{x})),
             \end{eqnarray*}
\begin{eqnarray*}\label{eq:4.9}
\pi_{\mathcal{E}^{'}}\left(a(x_{0},\underline{x})+\beta_{0}K_{0}(\underline{x})
+\sum_{i=1}^{r}\beta_{i}K_{i}\right)
             = \psi_{\mathcal{E}^{'}}(\pi(a(x_{0},\underline{x})))
             \end{eqnarray*}
             for $a\in \mathfrak{g}$, $\alpha_{i},\beta_{i}\in\mathbb{C}$, where
             we extend $\pi$ to
             $\tau[[x_{0}^{\pm 1},x_{1}^{\pm 1},\dots,x_{r}^{\pm 1}]]$
             canonically.
             Then
             \begin{eqnarray}\label{eq:4.10}
             \pi = \pi_{\widetilde{\mathcal{R}}}+ \pi_{\mathcal{E}^{'}}
             \end{eqnarray}
and the linear map $\varphi(u,v) = \pi_{\widetilde{\mathcal{R}}}(u)
                                         + \pi_{\mathcal{E}^{'}}(v)$ defines a representation
             of $\tau\oplus\tau$ on $W$. If $(W,\pi)$ is irreducible, we have that $W$ is an irreducible
$\tau\oplus\tau$-module. Furthermore,
$(W,\pi_{\widetilde{\mathcal{R}}})$ is a module in category
             $\widetilde{\mathcal{R}}$ and $(W,\pi_{{\mathcal{E}}^{'}})$
             is a module in category $\mathcal{E}_{\tau}^{'}$.
             At last, if $(W,\pi)$ is integrable, then $(W,\pi_{\widetilde{\mathcal{R}}})$
             is integrable in $\widetilde{\mathcal{R}}$ and $(W,\pi_{\mathcal{E}^{'}})$
             is integrable in $\mathcal{E}_{\tau}^{'}$.
\end{thm}
\begin{proof} Firstly, the relation (\ref{eq:4.10}) follows from Proposition
\ref{prop:4.8}. By the commutator relation (\ref{eq:2.3}) and Lemma
\ref{lem:4.11}, it is straightforward to check that
            $(W,\pi_{\widetilde{\mathcal{R}}})$ is a $\tau$-module and it is a restricted
            $\tau$-module in category $\widetilde{R}$.

            Let $f_{i}(x)$ for $i=0,1,\dots,r$ be nonzero polynomials satisfying that the nonzero roots of $f_{i}(x)$  for $1\le i\le r$ are
            multiplicity-free,  such that
            $f_{0}(x_{0})\pi(a(x_{0},\underline{x}))\in\mathcal{E}(W,r)$,
            $f_{i}(x_{i})\pi(a(x_{0},\underline{x}))=0$,
            $f_{i}(x_{i})\pi(K_{0}(\underline{x})) =0$
            for all $1\le i\le r$, $a\in \mathfrak{g}$.
            Then
            $$f_{0}(x_{0})\psi_{\widetilde{\mathcal{R}}}(\pi(a(x_{0},\underline{x})))
               = f_{0}(x_{0})\pi(a(x_{0},\underline{x})),\ \ \
              f_{i}(x_{i})\psi_{\widetilde{\mathcal{R}}}(\pi(a(x_{0},\underline{x})))=0$$
            for $1\le i\le r$, and
              $$f_{i}(x_{i})\psi_{\mathcal{E}}(\pi(a(x_{0},\underline{x})))=0$$
            for $0\le i\le r$. Thus we have
           \begin{eqnarray}\label{eq:4.31}
           f_{0}(x_{0})\pi_{\widetilde{\mathcal{R}}}((a(x_{0},\underline{x})))
            =f_{0}(x_{0})\psi_{\widetilde{\mathcal{R}}}(\pi(a(x_{0},\underline{x})))
            = f_{0}(x_{0})\pi(a(x_{0},\underline{x})),
            \end{eqnarray}
\begin{eqnarray}
             f_{i}(x_{i})\pi_{\widetilde{\mathcal{R}}}(a(x_{0},\underline{x})) =0
             \end{eqnarray}
            for $1\le i\le r$, and
            \begin{eqnarray}\label{eq:4.33}
            f_{i}(x_{i})\pi_{\mathcal{E}}(a(x_{0},\underline{x}))=0.
             \end{eqnarray}
for $0\le i\le r$.

            For $a,b\in \mathfrak{g}$, $w\in W$, by using the commutator relation (\ref{eq:2.3})
            and delta-function substitution property we have
 \begin{align*}
            &f_{0}(x_{0})[\pi_{\widetilde{\mathcal{R}}}(a(x_{0},\underline{x})), \pi_{\mathcal{E}}(b(y_{0},\underline{y}))]w \\
             =& f_{0}(x_{0})
             [\pi_{\widetilde{\mathcal{R}}}(a(x_{0},\underline{x})), \pi(b(y_{0},\underline{y}))]w
              -f_{0}(x_{0})
              [\pi_{\widetilde{\mathcal{R}}}(a(x_{0},\underline{x})), \pi_{\widetilde{\mathcal{R}}}(b(y_{0},\underline{y}))]w\\
             =& f_{0}(x_{0})
             [\pi(a(x_{0},\underline{x})), \pi(b(y_{0},\underline{y}))]w
              -f_{0}(x_{0})
              [\pi_{\widetilde{\mathcal{R}}}(a(x_{0},\underline{x})), \pi_{\widetilde{\mathcal{R}}}(b(y_{0},\underline{y}))]w\\
             =& f_{0}(x_{0})\pi([a,b](y_{0},\underline{y}))\prod_{i=0}^{r}x_{i}^{-1}\delta(\frac{y_{i}}{x_{i}})w
            + f_{0}(x_{0})\pi(K_{0}(\underline{y}))(\frac{\partial}{\partial y_{0}}y_{0}^{-1}\delta(\frac{x_{0}}{y_{0}}))
                 \prod_{i=1}^{r}x_{i}^{-1}\delta(\frac{y_{i}}{x_{i}})w   \\
            -& f_{0}(x_{0})\pi_{\widetilde{\mathcal{R}}}([a,b](y_{0},\underline{y}))
                  \prod_{i=0}^{r}x_{i}^{-1}\delta(\frac{y_{i}}{x_{i}})w
              - f_{0}(x_{0})\pi_{\widetilde{\mathcal{R}}}(K_{0}(\underline{y}))(\frac{\partial}{\partial y_{0}}y_{0}^{-1}\delta(\frac{x_{0}}{y_{0}}))\prod_{i=1}^{r}x_{i}^{-1}\delta(\frac{y_{i}}{x_{i}})w\\
             =& f_{0}(y_{0})\pi([a,b](y_{0},\underline{y}))\prod_{i=0}^{r}x_{i}^{-1}\delta(\frac{y_{i}}{x_{i}})w
             + f_{0}(x_{0})\pi(K_{0}(\underline{y}))(\frac{\partial}{\partial y_{0}}y_{0}^{-1}\delta(\frac{x_{0}}{y_{0}}))
                 \prod_{i=1}^{r}x_{i}^{-1}\delta(\frac{y_{i}}{x_{i}})w   \\
             -& f_{0}(y_{0})\pi_{\widetilde{\mathcal{R}}}([a,b](y_{0},\underline{y}))
                  \prod_{i=0}^{r}x_{i}^{-1}\delta(\frac{y_{i}}{x_{i}})w
             - f_{0}(x_{0})\pi_{\widetilde{\mathcal{R}}}(K_{0}(\underline{y}))(\frac{\partial}{\partial y_{0}}y_{0}^{-1}\delta(\frac{x_{0}}{y_{0}}))\prod_{i=1}^{r}x_{i}^{-1}\delta(\frac{y_{i}}{x_{i}})w\\
             =& 0,
            \end{align*}
            where we have used identity
           $f_{0}(y_{0})\pi([a,b](y_{0},\underline{y}))=
           f_{0}(y_{0})\pi_{\widetilde{\mathcal{R}}}([a,b](y_{0},\underline{y}))$
           from (\ref{eq:4.31}).

           Since $\pi_{\widetilde{\mathcal{R}}}(a(x_{0},\underline{x}))\in\mathcal{E}(W,r)$,
           we can multiply both sides by $l_{x_{0};0}\frac{1}{f_{0}(x_{0})}$
           to get
           \begin{eqnarray}\label{eq:4.34}
           [\pi_{\widetilde{\mathcal{R}}}(a(x_{0},\underline{x})), \pi_{\mathcal{E}}(b(y_{0},\underline{y}))]w=0.
           \end{eqnarray}
          Since $\pi_{\mathcal{E}}=\pi-\pi_{\widetilde{\mathcal{R}}}$,
          we have
          \begin{eqnarray*}
          [\pi_{\mathcal{E}}(a(x_{0},\underline{x})), \pi_{\mathcal{E}}(b(y_{0},\underline{y}))]w
        = [\pi(a(x_{0},\underline{x})), \pi(b(y_{0},\underline{y}))]w
          -[\pi_{\widetilde{\mathcal{R}}}(a(x_{0},\underline{x})),
          \pi_{\widetilde{\mathcal{R}}}(b(y_{0},\underline{y}))]w.
\end{eqnarray*}
          This together with the identity (\ref{eq:4.33}) shows that $(W,\pi_{\mathcal{E}})$ is
          a $\tau$-module in category $\mathcal{E_{\tau}}^{'}.$
          From (\ref{eq:4.34}), we see that $\varphi(u,v) = \pi_{\widetilde{\mathcal{R}}}(u)+
          \pi_{\mathcal{E}}(v)$ defines a representation of
          $\tau\oplus\tau$ on $W$. Since $\pi = \pi_{\widetilde{\mathcal{R}}}+
          \pi_{\mathcal{E}^{'}}$, it is easy to see that if $W$ is an irreducible $\tau$-module,
          then $W$ is an irreducible $\tau\oplus\tau$-module.

 Finally, we prove that $(W,\pi_{\widetilde{\mathcal{R}}})$ and
          $(W,\pi_{\mathcal{E}^{'}})$ are integrable when $(W,\pi)$ is integrable,
          i.e., we need to prove that for $a\in \mathfrak{g}_{\alpha}$ with $\alpha\in\Delta$
          and $n_{0}\in\mathbb{Z},\underline{n}\in\mathbb{Z}^{r}$,
          $\widetilde{a}(n_{0},\underline{n})$ and
          $\check{a}(n_{0},\underline{n})$ act locally nilpotently on $W$,
          where $\pi_{\widetilde{\mathcal{R}}}(a(n_{0},\underline{n}))=
                \widetilde{a}(n_{0},\underline{n})$
          and $\pi_{\mathcal{E}^{'}}(a(n_{0},\underline{n}))=
              \check{a}(n_{0},\underline{n})$.

          Since $[a,a]=0$ and $\langle a, a\rangle=0$, we have
          $[a(n_{0},\underline{n}),a(m_{0},\underline{m})]=0$
          for all $n_{0},m_{0}\in\mathbb{Z}$, $\underline{n},\underline{m}\in\mathbb{Z}^{r}$.
For $w\in W$, we have
          \begin{eqnarray*}
          &&{}a(n_{0},\underline{n})\widetilde{a}(x_{0},\underline{x})w\\
          &=&  a(n_{0},\underline{n})l_{x_{0};0}(1/f_{0}(x_{0}))(f_{0}(x_{0})a(x_{0},\underline{x})w)\\
          &=& l_{x_{0};0}(1/f_{0}(x_{0}))(f_{0}(x_{0})a(x_{0},\underline{x})a(n_{0},\underline{n})w)\\
          &=&
          \widetilde{a}(x_{0},\underline{x})a(n_{0},\underline{n})w.
          \end{eqnarray*}
          Thus
          $a(n_{0},\underline{n})\widetilde{a}(m_{0},\underline{m})
          =\widetilde{a}(m_{0},\underline{m})a(n_{0},\underline{n})$
          for all $n_{0},m_{0}\in\mathbb{Z}$, $\underline{n},\underline{m}\in\mathbb{Z}^{r}$.
Fix a vector $w\in W$. By Lemma \ref{lem:4.10},
          we know that
          $\widetilde{a}(n_{0},\underline{n})w
        =\sum_{i=0}^{l}\beta_{i}a(n_{0}+i,\underline{n})w $
        for some $l\in\mathbb{N},\beta_{1},\dots,\beta_{l}\in\mathbb{C}$.
Thus we get
        \begin{eqnarray}\label{eq:4.38}\widetilde{a}(n_{0},\underline{n})^{p}w=
        (\beta_{0}a(n_{0},\underline{n})+\cdots+\beta_{l}a(n_{0}+l,\underline{n}))^{p}w
        \end{eqnarray}
        for any $p\geq 0$. Since $(W,\pi)$ is an integrable $\tau$-module, there exists a positive
integer $k$
        such that
        $$a(m_{0},\underline{n})^{k}w = 0$$
        for $m_{0}=n_{0},n_{0}+1,\dots,n_{0}+l$.
        This together with \ref{eq:4.38} gives
        $\widetilde{a}(n_{0},\underline{n})^{k(l+1)}w=0$.
Since $\check{a}(n_{0},\underline{n}) =
         a(n_{0},\underline{n})-\widetilde{a}(n_{0},\underline{n})$
         and
         $[a(n_{0},\underline{n}),\widetilde{a}(n_{0},\underline{n})]=0$,
         we have
         \begin{eqnarray*}
         \check{a}(n_{0},\underline{n})^{k(l+2)}w &=&
         (a(n_{0},\underline{n})-\widetilde{a}(n_{0},\underline{n}))^{k(l+2)}w\\
         &=& \sum_{i\geq 0}\binom{k(l+2)}{i}(-1)^{i}
              a(n_{0},\underline{n})^{(k(l+2)-i)}\widetilde{a}(n_{0},\underline{n})^{i}w
              =0.
         \end{eqnarray*}
         So we have proved the integrability. This completes the
         proof.
\end{proof}

We now consider irreducible integrable modules in category
$\widetilde{\mathcal{R}}$. Let $W$ be an irreducible integrable
$\tau$-module in category $\widetilde{\mathcal{R}}$. Then there
exist nonzero polynomials $p_{i}(x)$ for
 $1\le i\le r$ satisfying that the nonzero
roots of $p_{i}(x)$ are multiplicity-free,
$p_{i}(x_{i})a(x_{0},\underline{x})w = 0$, and
$p_{i}(x_{i})K_{0}(\underline{x})w =0$.  Let $J=\langle p_{1}(t_{1}),\dots, p_{r}(t_{r})\rangle$ be the ideal generated by
$p_{i}(t_{i}),i=1,\dots,r.$ By the same argument as with
category $\mathcal{E}_{\tau}$, we see that a $\tau$-module $W$ in
category $\widetilde{\mathcal{R}}$ is a module for the Lie
algebra $(\mathfrak{g}\otimes \mathbb{C}[t_{0}^{\pm 1}]\oplus
\mathbb{C}K_{0})\otimes\mathbb{C}[t_{1}^{\pm 1},\dots,t_{r}^{\pm
1}]/J.$ We denote this quotient algebra by $\bar{\tau}$.

For each $1\leq i\leq r$, let $\{z_{ij}\in\mathbb{C}^{*},1\leq j\leq
M_{i}\}$ be the set of distinct nonzero roots of $p_{i}(x_{i})$,
where the positive integer $M_{i}$ is the number of distinct nonzero
roots of $p_{i}(x_{i})$. Set $M=M_{1}\cdots M_{r}$. Let $I_{i}=$
$(i_{1},\dots,i_{r})$, where $1\leq i_{j}\leq M_{j}$, $1\leq j\leq
r.$ Set $\underline{z}_{I_{i}}=(z_{1,i_{1}},\dots,z_{r,i_{r}})$
and
$\underline{z}_{I_{i}}^{\underline{n}}=z_{1,i_{1}}^{n_{1}}\cdots
z_{r,i_{r}}^{n_{r}}$ for $\underline{n} =
(n_{1},\dots,n_{r})\in\mathbb{Z}^{r}$. Let $\phi$ be the Lie
algebra homomorphism
$$ \phi:(\mathfrak{g}\otimes \mathbb{C}[t_{0}^{\pm 1}]\oplus
\mathbb{C}K_{0})\otimes\mathbb{C}[t_{1}^{\pm 1},\dots,t_{r}^{\pm 1}]
\rightarrow  \hat{\mathfrak{g}}_{M}:=({\mathfrak{g}\otimes
\mathbb{C}[t_{0}^{\pm 1}]\oplus \mathbb{C}K_{0}}){_{M}}
$$ defined by
$$ \phi(X\otimes t^{\underline{n}})=(\underline{z}_{I_{i}}^{\underline{n}}X)     $$
where $\hat{\mathfrak{g}}_{M}=
\oplus_{M-\mbox{copies}}(\mathfrak{g}\otimes
\mathbb{C}[t_{0}^{\pm}]\oplus \mathbb{C}K_{0})$,
$X\in\hat{\mathfrak{g}}={\mathfrak{g}\otimes \mathbb{C}[t_{0}^{\pm
1}]\oplus \mathbb{C}K_{0}}$. From the proof of Lemma 3.11 in
\cite{R}, we see that $\phi$ is surjective and Ker$\phi=J$, so that $\bar{\tau}\cong\hat{\mathfrak{g}}_{M}$. Consequently, $W$ is a module for
Lie algebra $\hat{\mathfrak{g}}_{M}$. Since $\tau$-module $W$ is
in category $\widetilde{\mathcal{R}}$, we also have
 $a(x_{0},\underline{x})\in\mathcal{E}(W,r)$ for $a\in \mathfrak{g}$,
i.e., $a(n_{0},\underline{n})w = 0 $ for $n_{0}$ sufficiently large.
Thus $W$ is an irreducible restricted
integrable $\hat{\mathfrak{g}}_{M}$-module. From \cite{DLM} (also see
the Theorem 2.4 of \cite{L1}), an irreducible restricted integrable
$\hat{\mathfrak{g}}$-module is an irreducible highest integrable
module.  Then by the above discussion and by Lemmas
2.6 and 2.7 in \cite{L1}, $W$ is isomorphic to a tensor product
module $V_{1}\otimes\cdots\otimes V_{M}$, where
$V_{1},\dots, V_{M}$ are irreducible highest integrable
$\hat{\mathfrak{g}}$-modules,
 with the action given by
$$ a(n_{0},\underline{n})(v_{1}\otimes\cdots\otimes v_{M})
 = \sum_{i=1}^{M}\underline{z}_{I_{i}}^{\underline{n}}
    (v_{1}\otimes\cdots\otimes a(n_{0})v_{i}\otimes\cdots\otimes v_{M})$$
    and
    $$K_{0}(\underline{n})(v_{1}\otimes\cdots\otimes v_{M})
    =\sum_{i=1}^{M}\underline{z}_{I_{i}}^{\underline{n}}
    (v_{1}\otimes\cdots\otimes K_{0}v_{i}\otimes\cdots\otimes v_{M}),$$
where $\underline{z}_{I_{i}}^{\underline{n}}$ is defined as above,
  $v_{i}\in V_{i}$ for $1\leq i\leq M$, and $a(n_{0})=a\otimes
  t_{0}^{n_{0}}.$
  We denote this module by $V_{1}(\underline{z}_{I_{1}})\otimes\cdots
 \otimes V_{M}(\underline{z}_{I_{M}}).$

\begin{thm}Let $\tau$ be the toroidal Lie algebra.
            Then every irreducible integrable $\tau$-module in category
            $\mathcal{C}_{\tau}$ is isomorphic to a module of the form
            $$V_{1}(\underline{z}_{I_{1}})\otimes\cdots\otimes V_{M}(\underline{z}_{I_{M}})
            \otimes U_{1}(\underline{z}_{1})\otimes\cdots\otimes
U_{s}(\underline{z}_{s}),$$ where each $V_{i}$ is an irreducible
             highest weight integrable
             $\hat{\mathfrak{g}}$-module,
              and each $U_{j}$ is a finite-dimensional
             irreducible $\mathfrak{g}$-module, $\underline{z}_{I_{j}}$  for $ j = 1,\dots,M$
             are defined above, and $\underline{z}_{i} =
(z_{0i},z_{1i},\dots,z_{ri})\in(\mathbb{C}^{*})^{r+1}$  for $1\le
i\le s$ are distinct $(r+1)$-tuples.
            \end{thm}

\begin{proof}Let $(W,\pi)$ be an irreducible integrable representation of $\tau$
            in category $\mathcal{C}_{\tau}$. By Theorem \ref{thm4.8}, $W$ is an irreducible
            $\tau\oplus\tau$-module with the action
            $\pi(u,v) = \pi_{\widetilde{\mathcal{R}}}(u)+ \pi_{\mathcal{E}^{'}}(v)$
            for $u,v\in\tau$.
           Moreover,
           $(W,\pi_{\widetilde{\mathcal{R}}})$
             is integrable in category $\widetilde{\mathcal{R}}$ and $(W,\pi_{\mathcal{E}^{'}})$
             is integrable in category $\mathcal{E}^{'}_{\tau}$. We
             view $W$ as a $\tau$-module in
             $\widetilde{\mathcal{R}}$, then $W$ is a restricted
             integrable $\tau$-module. From the above discussion, we
             see that $W$ is a restricted integrable
             $\hat{\mathfrak{g}}_{M}$-module. Furthermore, from the definition of Lie homomorphism $\phi$ and the proof of Lemma 3.11 in \cite{R},
             it follows that $W$ is a restricted integrable
             $\hat{\mathfrak{g}}$-module. From \cite{DLM} (also see
the Theorem 2.4 of \cite{L1}), $W$ is a direct sum of irreducible
integrable highest $\hat{\mathfrak{g}}$-modules.
 Thus by using the Lemma 2.7 of \cite{L1} with $A_{1}=A_{2}=U(\tau)$, and from the
 above discussion and Remark \ref{rmk:4.1}, we get our conclusion.
\end{proof}

\begin{rmk} Consider the toroidal Lie algebra
$\tau^{'}={\mathfrak{g}}\otimes {\mathbb{C}}[t_{0}^{\pm
1},\dots,t_{r}^{\pm 1}]\oplus\sum_{i=1}^{r}\mathbb{C}K_{i}$ and define
a category $\mathcal{C}_{\tau'}$ accordingly.
 Using similar arguments, we can show that every irreducible integrable $\tau^{'}$-module in category
 $\mathcal{C}_{\tau'}$ is isomorphic to a module of the form
 $U_{1}(\underline{z}_{1})\otimes\cdots\otimes
U_{s}(\underline{z}_{s})$, where  $U_{1},\dots,U_{s}$ are finite
dimensional irreducible $\mathfrak{g}$-modules and
$\underline{z}_{i} =
(z_{0i},z_{1i},\dots,z_{ri})\in(\mathbb{C}^{*})^{r+1}$  for $1\le
i\le s$ are distinct $(r+1)$-tuples.
\end{rmk}

\end{document}